\documentclass[10pt]{amsart}

\usepackage{amssymb, amsmath, amsthm, amsfonts}

\usepackage{mathrsfs}
\usepackage{extarrows}
\usepackage{amsthm}
\usepackage{amssymb}

\usepackage{mathrsfs}

\usepackage[all]{xy}

\usepackage{verbatim}
\usepackage{hyperref}
\hypersetup{
    colorlinks=true,       
    linkcolor=blue,          
    citecolor=blue,        
    filecolor=blue,      
    urlcolor=blue           
}

\newtheorem{theorem}{Theorem}[section]
\newtheorem{cor}[theorem]{Corollary}
\newtheorem{lem}[theorem]{Lemma}
\newtheorem{prop}[theorem]{Proposition}

\theoremstyle{definition}
\newtheorem{defin}[theorem]{Definition}

\theoremstyle{remark}

\newtheorem{rems}[theorem]{Remarks}

\newcommand{\EE}{\mathcal{E}}

\newcommand{\HH}{\mathrm{H}}
\newcommand{\End}{\mathrm{End}}
\newcommand{\Hom}{\mathrm{Hom}}

\newcommand{\bB}{\mathbb{B}}
\newcommand{\bP}{\mathbb{P}}

\newcommand{\bZ}{\mathbb{Z}}

\newcommand{\bR}{\mathbb{R}}
\newcommand{\bC}{\mathbb{C}}

\newcommand{\cA}{\mathcal{A}}
\newcommand{\cC}{\mathcal{C}}
\newcommand{\cO}{\mathcal{O}}

\newcommand{\cE}{\mathcal{E}}
\newcommand{\cF}{\mathcal{F}}

\newcommand{\cH}{\mathcal{H}}
\newcommand{\cI}{\mathcal{I}}
\newcommand{\cL}{\mathcal{L}}

\newcommand{\cV}{\mathcal{V}}

\newcommand{\arrow}{\rightarrow}

\newcommand{\Id}{\mathrm{Id}}

\newcommand{\Sym}{\mathrm{Sym}}
\newcommand{\Res}{\mathrm{Res}}
\newcommand{\sbt}{\,\begin{picture}(-1,1)(-1,-3)\circle*{3}\end{picture}\ }

\newcommand{\Gr}{\mathrm{Gr}}

\title[]{Semi-positivity from Higgs bundles}
\author{Yohan Brunebarbe}
\date{\today}

\begin{document}

\maketitle

\begin{abstract}
We prove a generalization of the Fujita-Kawamata-Zuo semi-positivity Theorem \cite{Fujita78, Kawamata81, Zuo_neg} for filtered regular meromorphic Higgs bundles and tame harmonic bundles. Our approach gives a new proof in the cases already considered by these authors. We give also an application to the geometry of smooth quasi-projective complex varieties admitting a semisimple complex local system with infinite monodromy.
 \end{abstract}

\section{Statement of the main results}

If $X$ is a smooth projective complex variety and $D = \cup_{i \in I} D_i \subset X$ is a simple normal crossing divisor, recall that a filtered regular meromorphic Higgs bundle $(\underline{\cE}, \theta)$ on $(X,D)$ consists of a parabolic bundle $\underline{\cE}$ together with a Higgs field $\theta$ which is logarithmic with respect to the parabolic structure along $D$, meaning that for any parabolic weight $\alpha \in {\bR}^I$ we have
$$\theta :  \cE^{\alpha} \arrow  \cE^{\alpha} \otimes \Omega^1_X(\log D). $$
See section \ref{non-abelian Hodge correspondence} for a reminder. We denote by $\cE^{\diamond}$ the vector bundle  in $\underline{\cE}$ with parabolic weight $(0, \cdots, 0) \in \bR^I$. \\

Our first result is a generalization of Fujita-Kawamata-Zuo semi-positivity Theorem \cite{Fujita78, Kawamata81, Zuo_neg, Brunebarbe_Crelle}, see Remarks \ref{rem_Fujita_Kawamata} for a discussion of former results. 
\begin{theorem}\label{main result}
Let $X$ be a smooth projective complex variety, $\cL$ be an ample line bundle on $X$ and $D \subset X$ be a simple normal crossing divisor. Let $(\underline {\cE}, \theta)$ be a filtered regular meromorphic Higgs bundle on $(X,D)$. Assume that $(\underline {\cE}, \theta)$ is $\mu_{\cL}$-polystable with vanishing parabolic Chern classes.
If $\cA$ is a subsheaf of $\cE^{\diamond}$ contained in the kernel of $\theta$, then its dual ${\cA}^{\star}$ is weakly positive.
\end{theorem}

The somewhat technical definition of a weakly positive torsion-free coherent sheaf, which is a generalisation of the notion of pseudo-effective line bundles for higher-rank sheaves, is recalled in section \ref{Weakly positive torsion-free sheaves}. We will show more precisely that the restricted base locus $\bB_-(\cA^{\star})$ of $\cA^{\star}$ is contained in the union of $ D$ and the locus in $X$ where $ \cA$ is not a locally split subsheaf of $\cE^{\diamond} $, cf. Theorem \ref{thm_restricted_base_locus}. \\

In the special case where the parabolic filtration is trivial, one has the following more precise result:

\begin{theorem}\label{trivial_parabolic_filtration}
Let $X$ be a smooth projective complex variety, $\cL$ be an ample line bundle on $X$ and $D \subset X$ be a simple normal crossing divisor. Let $(\cE, \theta)$ be a $\mu_{\cL}$-polystable logarithmic Higgs bundle on $(X,D)$ with vanishing Chern classes.
If $\cA$ is a locally split subsheaf of $\cE$ contained in the kernel of $\theta$, then its dual $\cA^{\star}$ is nef.
\end{theorem}
 
Let $X$, $\cL$ and $D$ as before, and set $j : U := X \backslash D \hookrightarrow X$ the inclusion. Given a $\mu_{\cL}$-polystable filtered regular meromorphic Higgs bundles $(\underline{\cE}, \theta)$ on $(X,D)$ with vanishing parabolic Chern classes, it follows from the works of Simpson \cite{Simpson_open}, Biquard \cite{Biquard} and Mochizuki \cite{Mochizuki_GeoTopo} that there exists an essentially unique tame pluriharmonic metric $h$ on $({\cE}^{\diamond}_{|U}, \theta)$ whose asymptotic behaviour is controlled by the parabolic vector bundle $\underline{\cE}$. The triple $({\cE}^{\diamond}_{|U}, \theta , h)$ defines a tame harmonic bundle on $U$, and ${\cE}^{\diamond}$ can be recovered from this data as the subsheaf of $j_*({\cE}^{\diamond}_{|U})$ whose sections have sub-polynomial growth with respect to $h$. Using this, we will see that Theorem \ref{main result} is a consequence of the following result:

\begin{theorem}\label{semi-negativity for harmonic bundles}
Let $X$ be a complex manifold, and $D \subset X$ be a simple normal crossing divisor. Set $U := X \backslash D$ and $j : U \hookrightarrow X$ the inclusion. \\
Let $(\cE, \theta , h)$ be a tame harmonic bundle on $U$, and $\cA$ be a subsheaf of $\cE$. The hermitian metric $h$ induces a (possibly singular) hermitian metric $h_{\cA}$ on $\cA$.
Denote by $\cA^{\diamond}$ the subsheaf of $j_* \cA$ whose sections have sub-polynomial growth with respect to $h_{\cA}$. Then $\cA^{\diamond}$ is a torsion-free coherent sheaf on $X$. Moreover, if $\cA$ is contained in the kernel of $\theta$, then there exists a unique singular hermitian metric with semi-negative curvature on $\cA^{\diamond}$ whose restriction to $U$ is $h_{\cA}$. 
\end{theorem}

For the notion of singular hermitian metric on a torsion-free coherent sheaf, the reader is referred to section \ref{Singular hermitian metrics on torsion-free sheaves}.
Theorem \ref{trivial_parabolic_filtration} follows then from the computation of the Lelong numbers of the induced singular hermitian metric on the dual sheaf $\cA^{\star}$, see Proposition \ref{prop_Lelong_numbers}. \\

An important example of tame harmonic bundle is given by the following construction. Let $X$ be a complex manifold, and $D \subset X$ be a simple normal crossing divisor. Let $(\cV, \nabla ,  \cF^{\sbt}, h)$ be a complex polarized variation of Hodge structures ($\bC$-PVHS) on $U := X \backslash D$, and denote by $(\cE, \theta)$ the associated Higgs bundle on $U$, i.e. $ \cE := \Gr_{\cF} \cV $ and $ \theta := \Gr_{\cF} \nabla $. If $h_{\cE}$ denotes the positive-definite hermitian metric on $\cE$ induced by $h$, then the triple $(\cE, \theta , h_{\cE})$ defines a tame harmonic bundle on $U$, cf. section \ref{PVHS}. As a direct consequence of Theorem \ref{semi-negativity for harmonic bundles}, we get the

\begin{theorem} \label{thm_PVHS}
Let $X$ be a complex manifold, and $D \subset X$ be a simple normal crossing divisor. Let $(\cV, \nabla ,  \cF^{\sbt}, h)$ be a log $\bC$-PVHS on $(X,D)$, and denote by $(\cE, \theta)$ the corresponding log-Higgs bundle. Assume that the eigenvalues of the residues of $(\cV, \nabla)$ along the irreducible components of $D$, which are known to be real numbers, are non-negative.\\
If $\cA$ is a subsheaf of $\cE$ contained in the kernel of the Higgs field $\theta$, then the (possibly singular) hermitian metric $h_{\cA}$ on $\cA_{|U}$ induced by $h_{\cE}$ extends uniquely as a singular hermitian metric on $\cA$ with semi-negative curvature. 
\end{theorem}

\begin{cor} \label{cor_PVHS}
In the situation of Theorem \ref{thm_PVHS}, if moreover $X$ is projective, then the dual $\cA^\star$ of $\cA$ is weakly positive.
\end{cor}

\begin{rems} \label{rem_Fujita_Kawamata}
\begin{enumerate}
\item In the situation of Corollary \ref{cor_PVHS}, if one assumes moreover that the residues of $(\cV, \nabla)$ are nilpotent and that $\cA$ is a locally split subsheaf of $\cE$ contained in the kernel of the Higgs field $\theta$, then one can apply Theorem \ref{trivial_parabolic_filtration} to get that its dual $\cA^\star$ is nef. This special case is already known, cf. \cite{Zuo_neg} and \cite[Theorem 0.6]{Brunebarbe_Crelle}. See also \cite{Fujino-Fujisawa2017} for a different proof. From this, one can easily deduce Corollary \ref{cor_PVHS} in the special case where the eigenvalues of the residues are rational numbers, as explained for example in \cite{Popa_Wu}.

\item With the notations of Theorem \ref{thm_PVHS}, let $p$ be the biggest integer such that $\cF^p \cV = \cV$. It follows from Griffiths’ transversality that $Gr^p_{\cF} \cV = \cV \slash \cF^{p+1} \cV$ is a subbundle of $\cE$ contained
in the kernel of the Higgs field, so that its dual is weakly positive (note that we make no assumption on the monodromy at infinity). When the monodromy at infinity is unipotent, $(Gr^p_{\cF} \cV )^{\star}$ is canonically isomorphic to the lowest piece of the Hodge filtration of the dual log $\bC$-PVHS. Therefore, Theorem \ref{trivial_parabolic_filtration} implies that the lowest piece of the Hodge filtration of a log $\bC$-PVHS with unipotent monodromy at infinity is nef, as previously shown for \textit{real} polarized variation of Hodge structures by Fujita \cite{Fujita78} and Zucker \cite{Zucker82} for curves, and by Kawamata \cite{Kawamata81}, Fujino-Fujisawa \cite{Fujino-Fujisawa} and Fujino-Fujisawa-Saito \cite{Fujino-Fujisawa-Saito} in higher dimensions. An algebraic proof of this last result appears in \cite[Corollary 1.4]{Arapura16}.
\end{enumerate}
\end{rems}

Theorem \ref{main result} has some consequences for the geometry of smooth quasi-projective complex varieties admiting a complex local system with infinite monodromy that we now discuss. In analogy with \cite{Kollar93, Kollar_book_95} we say that a complex local system $L$ on a smooth quasi-projective complex variety $X$ is generically large when there exists countably many closed subvarieties $D_i \subsetneq X$ such that for every smooth quasi-projective complex variety $Z$ equipped with a proper map $f : Z \arrow X$ satisfying $f(Z) \not \subset \cup D_i$ the pullback local system $f^\ast L$ is non-trivial. This condition is less exceptional that one might first thinks, since conjecturally at least any semisimple complex local system is virtually the pull-back of a generically large local system.

\begin{theorem}\label{generically large local system}
Let $X$ be a smooth projective complex variety and $D \subset X$ be a simple normal crossing divisor. If there exists a generically large semisimple complex local system with discrete monodromy on $X \backslash D$, then the logarithmic cotangent bundle $\Omega^1_X(\log D)$ of $(X,D)$ is weakly positive.
\end{theorem}

Under the same assumptions, if the local system underlies a complex polarized variation of Hodge structures, then it is known that $(X,D)$ is of log-general type, cf. \cite{Zuo_neg, Brunebarbe_Crelle, Brunebarbe-Cadorel}. More generally, we conjecture that $(X,D)$ is of log-general type if there exists a generically large semisimple complex local system on $X \backslash D$ whose monodromy representation is Zariski-dense in a semisimple complex algebraic group, as already known when $D = \emptyset$ \cite{Mok92, Zuo96, CCE15}.

\section{A review of the non-abelian Hodge correspondence for quasi-projective varieties}\label{non-abelian Hodge correspondence}

In this section, we give a brief review of the non-abelian Hodge correspondence for smooth quasi-projective complex varieties, which is one of the key input of our approach. References include \cite{Simpson_open, Biquard, Mochizuki_asterisque309, Mochizuki_GeoTopo, Mochizuki_Wild-harmonic-bundles, Sabbah05}.

\subsection{Parabolic bundles} We follow \cite{Mochizuki_asterisque309} and \cite{Iyer_Simpson}, however be aware that our terminology differs slightly from theirs.\\

Let $X$ be a complex manifold and $D = \cup_{i \in I} D_i \subset X$ be a simple normal crossing divisor. A \textit{parabolic sheaf} $\underline{\cE}$ on $(X,D)$ is a meromorphic bundle $\cE$ on $(X,D)$, i.e. a torsion-free coherent $\cO_X\left[*D\right]$-module, endowed with a collection of torsion-free coherent $\cO_X$-submodules ${\cE}^{\alpha}$ indexed by multi-indices $\alpha = (\alpha_i)_{i \in I}$ with $\alpha_i \in \bR$, satisfying the following conditions:

\begin{itemize}
\item (the filtration is exhaustive and decreasing) $\cE = \cup_{\alpha} \cE^{\alpha}$ and $\cE^{\alpha} \hookrightarrow \cE^{\beta}$ whenever $\alpha \geq \beta$ (i.e. $\alpha_i \geq \beta_i$ for all $i$),

\item (normalization/support) $\cE^{\alpha + \delta^i} = \cE^{\alpha}(- D_i)$, where $\delta^i$ denotes the multi-index $\delta^i_i =1$ and $\delta^i_j =0$ for $j \neq i$,
\item (semicontinuity) for any given $\alpha$ there exists $c > 0$ such that for any multi-index $\epsilon$ with $0 \leq \epsilon_i < c$ we have $\cE^{\alpha - \epsilon} = \cE^{\alpha}$.
\end{itemize}

We denote by $\cE^{\diamond}$ the subsheaf of $\cE$ with parabolic weight $(0, \cdots, 0) \in \bR^I$.\\

Fixing $c \in \bR$, it follows from the axioms that a parabolic sheaf is determined by the $\cE^{\alpha}$ for the collection of jumping indices $\alpha$ with $c \leq \alpha_i < c +1$ for any $c \in \bR$, and this collection is finite when $D$ has a finite number of irreducible components.\\ 

We say that $\underline{\cE}$ is \textit{locally abelian} if in a Zariski neighborhood of any point $x \in X$, $\underline{\cE}$ is isomorphic to a direct sum of parabolic line bundles (i.e. parabolic sheaves which are locally-free of rank $1$). A \textit{parabolic bundle} on $(X,D)$ is a parabolic sheaf wich is locally abelian. In particular, all the $\cE^{\alpha}$ are locally-free. The parabolic structure is called trivial when all coefficients of the jumping indices are integers. \\

There is a notion of parabolic Chern classes for parabolic bundles (we refer the reader to \cite[chapter 3]{Mochizuki_asterisque309} and \cite{Iyer_Simpson} for the quite complicated formulas). When $\underline{\cE}$ is a parabolic bundle with trivial parabolic structure, then its parabolic Chern classes coincide with those of $\cE^{\diamond}$.

\subsection{Prolongation according to growth conditions}
Let $X$ be an $n$-dimensional complex manifold, and let $D  = \cup_{ i \in I}  D_i$ be a simple normal crossing divisor.

\begin{defin} Let $P$ be a point of $X$, and let $D_1, \cdots, D_l$ be the irreducible components of $D$ containing $P$. An admissible coordinate system around $P$
is a pair $(U, \psi)$, where:
\begin{itemize}
\item $U$ is an open subset of $X$ containing $P$,
\item $\psi$ is a holomorphic isomorphism $U \simeq \Delta^n = \left\{ (z_1, \cdots , z_n) \in \bC^n ~\middle|~ |z_i| < 1 \right\}$ such that $\psi(P) = (0, \cdots , 0)$ and $\psi(D_i) = \left\{z_i = 0 \right\}$ for any $i = 1, \cdots , l$.
\end{itemize}
\end{defin}

Let $\cE$ be a holomorphic vector bundle on $X \backslash D$ equipped with a hermitian metric $h$.
Let $\alpha = (\alpha_i)_{ i \in I} \in \bR^I$ be a tuple of real numbers.

\begin{defin} Let $U$ be an open subset of $X$, and $s$ be an element
of $\HH^0(U \backslash D, \cE)$. We say that the order of growth of $s$ is bigger than $\alpha$ if for any point $P \in U$, any admissible coordinate system $(U, \psi)$ around $P$ and any positive real number $\epsilon >0$, there exists a
positive constant $C$ such that the following inequality holds on $U \backslash D$:
\[ \Vert s(z) \Vert_h \leq  C \cdot  \prod_{i=1}^l |z_i|^{\alpha_i - \epsilon} \]
\end{defin}

We denote by $\cE^{\alpha}$ the subsheaf of $j_* \cE$ of sections with order of growth bigger than $\alpha$. The $\cE^{ \alpha}$, $\alpha \in \bR^I$, form a decreasing filtration of the subsheaf of $j_* \cE$. 

\begin{defin}[Moderate hermitian metric on a holomorphic vector bundle]
Let $X$ be a complex manifold, $D \subset X$ be a simple normal crossing divisor, and $\cE$ be a holomorphic bundle on $X \backslash D$. Set $j: X \backslash D \hookrightarrow X$ the inclusion. A hermitian metric $h$ on $\cE$ is called moderate if the subsheaf $\cup_{\alpha \in \bR^I} \cE^{ \alpha}$ of $j_* \cE$ consisting of sections with moderate growth near $D$ is a meromorphic sheaf (i.e. a torsion-free coherent $\cO_X\left[*D\right]$-module), and for every $\alpha \in \bR^I$ the subsheaf $\cE^{\alpha}$ of $j_* \cE$ is a coherent $\cO_X$-module.
\end{defin}

When the metric is moderate, one verifies immediately that the collection of torsion-free coherent sheaves $\cE^{\alpha}$ form a parabolic sheaf on the pair $(X,D)$.

\subsection{$\lambda$-Connection bundles}

\begin{defin}[$\lambda$-Connection]
Let $\lambda \in \bC$.
Let $X$ be a complex manifold and $\cE$ be a locally-free $\cO_X$-module of finite rank. A $\lambda$-connection on $\cE$ is a $\bC$-linear map $D^{\lambda} : \cE \arrow \cE \otimes_{\cO_X} \Omega_X^1 $ which satisfies the twisted Leibniz rule:
$$ D^{\lambda} (f \cdot s) = f \cdot D^{\lambda} (s) + \lambda \cdot df \wedge s, $$ 
where $f$ and $s$ are sections of $\cO_X$ and $\cE$ respectively. 
\end{defin}

One verifies immediately that a $\lambda$-connection $D^{\lambda}$ induces for every integer $p \geq 0$ a $\bC$-linear map $D^{\lambda} : \cE  \otimes_{\cO_X} \Omega_X^p \arrow \cE \otimes_{\cO_X} \Omega_X^{p+1} $. A $\lambda$-connection $D^{\lambda}$ is called \textit{flat} or \textit{integrable} if $D^{\lambda} \circ D^{\lambda} =0$. 

\begin{defin}[Higgs bundle]
A pair $(\cE, \theta)$ consisting of a locally-free $\cO_X$-module of finite rank $\cE$ equipped with a flat $\lambda$-connection $\theta$ is called a Higgs bundle when $\lambda = 1$.
\end{defin}

\begin{defin}[Logarithmic $\lambda$-connection]
Let $X$ be a complex manifold and $D \subset X$ be a normal crossing divisor. A logarithmic $\lambda$-connection bundle on $(X,D)$ is a pair $(\cE, D^{\lambda})$ consisting of a locally-free $\cO_X$-module of finite rank $\cE$ together with a $\bC$-linear map $D^{\lambda} : \cE \arrow \cE \otimes_{\cO_X} \Omega_X^1( \log D) $ which satisfies the twisted Leibniz rule.
\end{defin}

\begin{defin}[Meromorphic $\lambda$-connection]
Let $X$ be a complex manifold and $D \subset X$ be a normal crossing divisor. A filtered regular meromorphic $\lambda$-connection bundle on $(X,D)$ is a pair $(\underline{\cE}, D^{\lambda})$ consisting of a parabolic bundle $\underline{\cE}$ together with a $\bC$-linear map $D^{\lambda} : \cE \arrow \cE \otimes_{\cO_X\left[*D\right]} \Omega_X^1\left[*D\right] $ which satisfies the twisted Leibniz rule, and such that for any parabolic weight $\alpha$ the $\lambda$-connection $D^{\lambda}$ induces a structure of logarithmic $\lambda$-connection bundle on $(\cE^{\alpha},  D^{\lambda})$.
\end{defin}

The notion of flatness generalizes to the logarithmic and the meromorphic settings.
In particular, by specializing to $\lambda = 0$ we get the notions of logarithmic Higgs bundles and filtered regular meromorphic Higgs bundles on $(X,D)$.

\begin{defin}[Residues of a logarithmic $\lambda$-connection]
Let $X$ be a complex manifold, $D \subset X$ be a normal crossing divisor and $(\cE, D^{\lambda})$ be a logarithmic $\lambda$-connection bundle on $(X,D)$. For any irreducible component $D_k$ of $D$, there is an associated Poincar\'e residue $\cO_X$-linear map $R_k : \Omega^1_X(\log D) \arrow \cO_{D_k}$. The map $(R_k \otimes \text{id}) \circ D^{\lambda}$ induces an $\cO_{D_k}$-linear endomorphism $\text{res}_{D_k}(D^{\lambda}) \in \End(\cE_{|D_k})$ which is called the residue of $D^{\lambda}$ along $D_k$. 
\end{defin}

If $X$ is a projective manifold equipped with an ample line bundle $\cL$, there is a corresponding notion of $\mu_{\cL}$-stability for logarithmic $\lambda$-connection bundles and filtered regular meromorphic $\lambda$-connection bundles on $(X,D)$, cf. \cite[section 3.1.3]{Mochizuki_asterisque309}. When $(\underline{\cE}, D^{\lambda})$ is a filtered regular meromorphic $\lambda$-connection bundle on $(X,D)$ such that the parabolic structure on $\underline{\cE}$ is trivial, then its $\mu_{\cL}$-stability is equivalent to the $\mu_{\cL}$-stability of the logarithmic $\lambda$-connection bundle $(\cE^{\diamond}, D^{\lambda})$.

\subsection{Pluriharmonic metrics}

Let $(\cE, \theta)$ be a Higgs bundle on a complex manifold $X$, and $h_{\cE}$ be a hermitian metric on $\cE$. Let $E := \cC^{\infty}_X \otimes_{\cO_X} \cE$ be the smooth vector bundle underlying $\cE$, and $ \overline{\partial}_{\cE}$ be the operator defining the holomorphic structure on $\cE$. Let $\nabla^u = \partial_{\cE} + \overline{\partial}_{\cE}$ be the Chern connection associated to the hermitian metric $h_{\cE}$ on the holomorphic bundle $\cE$, and $\theta^{\star}$ be the adjoint of $\theta$ with respect to $h_{\cE}$. By definition, the metric $h_{\cE}$ is said \textit{pluriharmonic} if the operator $\nabla := \partial_{\cE} + \overline{\partial}_{\cE} + \theta + \theta^{\star}$ is integrable, i.e. if the differentiable form $\nabla^2 \in A^2(\End(E))$ is zero. In that case, the holomorphic bundle $\mathscr{E} := (E,  \overline{\partial}_E + \theta^{\star} )$ equipped with the connection $\nabla$ defines a flat bundle.\\

Conversely, let $(\mathscr{E}, \nabla)$ be a flat bundle on a complex manifold $X$, and denote by $E$ the smooth vector bundle underlying $\mathscr{E}$. The choice of a hermitian metric $h_{\mathscr{E}}$ on $\mathscr{E}$ induces a canonical decomposition $\nabla = \nabla^u + \Psi$, where $\nabla^u$ is a unitary connection on $\mathscr{E}$ with respect to $h_{\mathscr{E}}$ and $\Psi$ is autoadjoint for $h_{\mathscr{E}}$. Both decompose in turn in their components of type $(1, 0)$ and $(0, 1)$:
\[ \nabla^u = \partial_{\cE} + \overline{\partial}_{\cE}, \hspace{ 1 cm} \Psi = \theta + \theta^{\star}.\]

By definition, the metric $h_{\mathscr{E}}$ is said \textit{pluriharmonic} if the operator $D^{\prime \prime}: =  \overline{\partial}_{\cE} + \theta$ is integrable, i.e. if the differentiable form $(D^{\prime \prime})^2 \in A^2(\End(E))$ is zero. In that case, the holomorphic bundle $\cE := (E,  \overline{\partial}_E)$ endowed with the one-form $\theta \in A^1(\End(E))$ defines a Higgs bundle.\\

More generally, there is a notion of pluriharmonicity for metrics on a holomorphic vector bundle equipped with a flat $\lambda$-connection, see \cite[section 2.2.1]{Mochizuki_GeoTopo}.

\subsection{Tame harmonic bundles}

\begin{defin}[Harmonic bundle]
A harmonic bundle on a complex manifold is a Higgs bundle endowed with a pluriharmonic metric, or equivalently a flat bundle equipped with a pluriharmonic metric.
\end{defin}

\begin{defin}[Tame harmonic bundle]
Let $X$ be a complex manifold and $D \subset X$ a simple normal crossing divisor. A harmonic bundle $(\cE, \theta, h)$ on $U := X \backslash D$ is called tame if there exists a logarithmic Higgs bundle on $(X,D)$ extending $(\cE, \theta)$. 
\end{defin}

In that case, the characteristic polynomial of the Higgs field $\theta$, which is a polynomial whose coefficients are holomorphic symmetric differential forms on $U$, extends to $X$ as an element of $\Sym( \Omega^1_X(\log D)) [T]$. This property turns out to be sufficient to show that a harmonic bundle $(\cE, \theta, h)$ on $U := X \backslash D$ is tame. Note that the element of $\Sym( \Omega^1_X(\log D)) [T]$ does not depend on the extension, because its restriction to $U$ is fixed. In particular, the eigenvalues of the residues are independent of the extension.

\begin{theorem}[Simpson, Mochizuki, cf. {\cite[Proposition 2.53]{Mochizuki_GeoTopo}}]\label{tame harmonic metric are moderate}
Let $X$ be a complex manifold and $D \subset X$ be a simple normal crossing divisor. If $(\mathscr{E}, \nabla, h)$ (resp. $(\cE, \theta, h)$) is a tame harmonic bundle on $ X \backslash D$, then the metric $h$ is moderate and the corresponding filtration according to the order of growth on $\mathscr{E}$ (resp. $\cE$) defines a parabolic bundle on $(X,D)$.
\end{theorem}

\subsection{Deligne-Manin filtration}\label{Deligne-Manin filtration}
Let $X$ be a complex manifold and $D = \cup_I D_i\subset X$ be a simple normal crossing divisor. Let $(\mathscr{E}, \nabla)$ be a meromorphic bundle on $(X,D)$ equipped with a regular meromorphic connection. For every $\alpha \in {\bR}^I$, let $\mathscr{E}^{\alpha}$ be the unique locally-free $\cO_X$-module of finite rank contained in $\mathscr{E}$ such that $\nabla$ induces a connection with logarithmic singularities $\nabla : \mathscr{E}^{\alpha} \arrow \mathscr{E}^{\alpha} \otimes \Omega^1_X( \log D)$ such that the real part of the eigenvalues of the residue of $\nabla$ along $D_i$ belongs to $\left[ \alpha_i , \alpha_i + 1 \right)$, cf. \cite[Proposition 5.4]{Deligne-singuliers-reguliers}. We call this filtration Deligne-Manin filtration. The meromorphic bundle $\mathscr{E}$ equipped with this filtration defines a parabolic bundle $\underline{\mathscr{E}}$, and $(\underline{\mathscr{E}}, \nabla)$ is a filtered regular meromorphic connection bundle on $(X,D)$.\\

Let now $(\mathscr{E}, \nabla, h)$ be a tame harmonic bundle on $U := X \backslash D$. There exists a unique meromorphic bundle on $(X,D)$ equipped with a regular meromorphic connection extending $(\mathscr{E}, \nabla)$ (the so-called Deligne's extension). When equipped with the Deligne-Manin filtration introduced above, this defines a canonical filtered regular meromorphic connection bundle $(\underline{\mathscr{E}}^{DM}, \nabla)$ on $(X,D)$ which extends $(\mathscr{E}, \nabla)$. We have also the filtered regular meromorphic connection bundle $(\underline{\mathscr{E}}^{h}, \nabla)$ on $(X,D)$ defined using the filtration according to the order of growth with respect to $h$, cf. Theorem \ref{tame harmonic metric are moderate}.
The following result is well-known and is a direct consequence of the table in \cite[p.720]{Simpson_open}.

\begin{lem}\label{DM=h} Let $X$ be a complex manifold and $D  \subset X$ be a simple normal crossing divisor. Let $ (\mathscr{E}, \nabla, h)$ be a tame harmonic bundle on $X \backslash D$, and $(\cE, \theta)$ be the corresponding Higgs bundle. Then the following properties are equivalent:
\begin{itemize}
\item The eigenvalues of the residues of $\theta$ along the irreducible components of $D$ are purely imaginary.
\item The filtered regular meromorphic connection bundles $(\underline{\mathscr{E}}^{DM}, \nabla)$ and  $(\underline{\mathscr{E}}^{h}, \nabla)$ are canonically isomorphic.
\end{itemize}	
\end{lem}

\subsection{The non-abelian Hodge correspondence in the non-compact case}

Let $X$ be a smooth irreducible complex projective variety equipped with an ample line bundle $\cL$, and $D$ be a simple normal crossing divisor of $X$.
\begin{theorem}[Mochizuki, cf. {\cite[Theorem 1.1]{Mochizuki_GeoTopo}}]\label{Mochizuki correspondence}
Fix $\lambda \in \bC$. Let $(\underline{\cE}, D^{\lambda})$ be a flat filtered regular meromorphic $\lambda$-connection bundle on $(X,D)$. Then the following conditions are equivalent:
\begin{enumerate}
\item  $(\underline{\cE}, D^{\lambda})$ is $\mu_{\cL}$-polystable with vanishing first and second parabolic Chern classes.
\item There exists a pluriharmonic metric on $(\cE_{|X \backslash D},  D^{\lambda})$ adapted to the parabolic structure.
\end{enumerate}
Such a metric is unique up to obvious ambiguity.
\end{theorem}

In particular, this induces a natural equivalence of categories between the category of $\mu_{\cL}$-polystable flat filtered regular meromorphic connection bundles on $(X,D)$ with vanishing first and second parabolic Chern classes and the category of $\mu_{\cL}$-polystable filtered regular meromorphic Higgs bundles on $(X,D)$ with vanishing first and second parabolic Chern classes. This equivalence preserves tensor products, direct sums and duals.

\section{Singular hermitian metrics on torsion-free sheaves}\label{Singular hermitian metrics on torsion-free sheaves}
We recall here the basic definitions concerning singular hermitian metrics on torsion-free sheaves, after \cite{Berndtsson-Paun, Paun16, Paun-Takayama, Hacon-Popa-Schnell}. We follow very closely the presentation of \cite{Hacon-Popa-Schnell}, to which the reader is referred for more details.

\subsection{Singular hermitian inner products}
\begin{defin}
A singular hermitian inner product on a finite-dimensional complex
vector space $V$ is a function $\Vert \cdot \Vert : V  \arrow [0, + \infty ]$ with the following properties:
\begin{enumerate}
\item $\Vert \lambda \cdot v \Vert = |\lambda| \cdot \Vert v \Vert$ for every $\lambda \in \bC \backslash \{0\}$ and every $v \in V$, and $\Vert 0 \Vert = 0$.
\item $\Vert v + w \Vert  \leq \Vert v \Vert + \Vert w \Vert$ for every $v, w \in V$, where by convention $\infty \leq \infty$.
\item $\Vert v + w \Vert^2 + \Vert v - w \Vert^2 = 2 \cdot \Vert v \Vert^2 + 2 \cdot \Vert w \Vert^2$ for every $v, w \in V$.
\end{enumerate}
\end{defin}

Let $V_0$ (resp. $V_{\mathrm{fin}}$) be the subset of $V$ of vectors with zero (resp. finite) norm. It follows easily from the axioms that both $V_0$ and $V_{\mathrm{fin}}$ are linear subspaces of $V$. By definition, $\Vert \cdot \Vert$ is said positive definite if $V_0 = 0$ and finite if $V_{\mathrm{fin}} = V$.\\

A singular hermitian inner product $\Vert \cdot \Vert$ on $V$ induces canonically a singular hermitian inner product $\Vert \cdot \Vert^{\star}$ on its dual $V^{\star} = \Hom_{\bC}(V, \bC)$ by setting
\[ \Vert f \Vert^{\star} := \sup \left\{ \frac{|f(v)|}{\Vert v \Vert} ~\middle|~ v \in V  ~ \text{with}  ~ \Vert v \Vert \neq 0 \right\} \]

for any linear form $f \in V^{\star}$, with the understanding that a fraction with
denominator $+ \infty$ is equal to $0$. (If $V_0 = V$, then we define $\Vert f \Vert^{\star}  = 0 $ for $f = 0$, and
$\Vert f \Vert^{\star}  = + \infty $ otherwise.)
Note however that in general there is no induced singular hermitian inner product on $\det V$.\\   

Observe that the dual of a finite singular hermitian inner product is positive definite, and conversely.

\subsection{Singular hermitian metrics on vector bundles}
Let $X$ be a connected complex manifold, and let $\cA$ be a non-zero holomorphic vector bundle on $X$.

\begin{defin}(compare \cite[Definition 17.1]{Hacon-Popa-Schnell} and \cite[Definition 2.8]{Paun16})
A singular hermitian metric on $\cA$ is a function $h$ that associates to every point $x \in X$ a singular hermitian inner product $\Vert \cdot \Vert_{h,x} : \cA_x \arrow [0, + \infty ]$ on
the complex vector space $\cA_x$, subject to the following two conditions:
\begin{itemize}
\item $h$ is finite and positive definite almost everywhere, meaning that for all $x$
outside a set of measure zero, $\Vert \cdot \Vert_{h,x}$ is a hermitian inner product on $\cA_x$.
\item $h$ is measurable, meaning that the function
\[  \Vert s \Vert_{h,x} : U \arrow [0, + \infty ], x \mapsto \Vert s(x) \Vert_{h,x} \]
is measurable whenever $U \subset X$ is open and $s \in \HH^0(U, \cA)$.
\end{itemize}
\end{defin}

\begin{defin} A singular hermitian metric $h$ on $\cA$ has semi-negative curvature if the
function $x \mapsto \log  \Vert s(x) \Vert_{h,x}$ is plurisubharmonic for every local section $s \in \HH^0(U, \cA)$. It has semi-positive curvature if its dual metric $h^{\star}$ on $\cA^{\star}$ has semi-negative curvature.
\end{defin}

A priori, a singular hermitian metric $h$ is only defined almost everywhere, because its coefficients are measurable functions. However, if $h$ has semi-positive or semi-negative curvature, then the singular hermitian inner product $\Vert \cdot \Vert_{h,x}$ is unambiguously defined at each point of $X$. Note also that, if $h$ is a singular hermitian metric with semi-negative curvature, then the singular hermitian inner product $\Vert \cdot \Vert_{h,x}$ is finite at each point of $X$.  
As a converse, we have the:

\begin{lem}\label{extension of metric with semi-negative curvature} Let $\cA$ be a non-zero holomorphic vector bundle on a connected complex manifold $X$. Let $U$ be an open dense subset of $X$, and $h$ be a singular hermitian metric with semi-negative curvature on $\cA_{|U}$. The following two assertions are equivalent:
\begin{enumerate}
\item $h$ is a restriction to $U$ of a singular hermitian metric with semi-negative curvature on $\cA$,
\item the function $x \mapsto \log  \Vert s(x) \Vert_{h,x}$ is locally bounded from above in the neighborhood of every point of $X$ for every local section $s \in \HH^0(U, \cA)$.
\end{enumerate}
Moreover, the metric extending $h$ is unique when it exists.
\end{lem}
\begin{proof}
This is a direct application of Riemann extension theorem for plurisubharmonic functions.
\end{proof}

\subsection{Singular hermitian metrics on torsion-free sheaves}
Let $X$ be a connected complex manifold, and let $\cA$ be a non-zero torsion-free coherent sheaf on $X$. Let $X_{\cA} \subset X$ denote the maximal open subset where $\cA$ is locally free, so that $X \backslash X_{\cA}$ is a closed analytic
subset of codimension at least $2$.

\begin{defin} A singular hermitian metric on $\cA$ is a singular hermitian metric
$h$ on the holomorphic vector bundle $\cA_{| X_{\cA}}$. \\
We say that such a metric has semi-positive curvature if the pair $(E, h)$ has semi-positive curvature.
\end{defin}

\section{Weakly positive torsion-free sheaves}\label{Weakly positive torsion-free sheaves}

For the reader convenience, we recall in this section the notion of weak positivity for torsion-free coherent sheaves, due to Viehweg and later refined by Nakayama.

\begin{defin}
Let $X$ be a complex quasi-projective scheme. A coherent sheaf $\cA$ on $X$ is globally generated at a point $x \in X$ if the natural map $\HH^0(X, \cA) \otimes_{\bC} \cO_X \arrow \cA$ is surjective at $x$.
\end{defin}

\begin{defin}[Nakayama,  {\cite[Definition V.3.20]{Nakayama}}]
Let $\cA$ be a torsion-free coherent sheaf on a smooth projective complex variety  $X$. We say that $\cA$ is weakly positive at a point $x \in X$ if for every ample invertible sheaf $\cH$ on $X$ and every positive integer $\alpha  > 0$ there exists an integer $\beta > 0$ such that $\widehat{S}^{\alpha \cdot \beta} \cA \otimes_{ \cO_X} {\cH}^{\beta}$ is globally generated at $x$.\\
Here the notation $\widehat{S}^{k} \cA $ stands for the reflexive hull of the sheaf $S^{k} \cA $, i.e. $\widehat{S}^{k} \cA = i_* ( S^{k} i^*\cA)$ where $i : X_{\cA} \hookrightarrow X$ is the inclusion of the maximal open subset where $\cA$ is locally free.
\end{defin}

\begin{defin}(Restricted base locus)
Let $\cA$ be a torsion-free sheaf on a smooth projective complex variety  $X$. The restricted base locus $\bB_-(\cA)$ of $\cA$ is the subset of $X$ of points at which $\cA$ is weakly positive. It is a priori only a union of closed subvarieties of $X$.
\end{defin}

\begin{defin}
A torsion-free coherent sheaf $\cA$ is weakly positive in the sense of Viehweg if there exists a dense open subset $U \subset X$ such that $\cA$ is weakly positive at every $x \in U$, or equivalently when $\bB_-(\cA)$  is not Zariski-dense.
\end{defin}

Note that if $\cA$ is locally free, then $\cA$ is nef if and only if it is weakly positive at every $x \in X$. \\

The following result gives an analytical criterion to show that a sheaf is weakly positive.

\begin{theorem}[P{\u a}un-Takayama, cf. {\cite[Theorem 2.21]{Paun16}} and {\cite[Theorem 2.5.2]{Paun-Takayama}}] \label{Paun-Takayama} Let $X$ be a smooth projective variety, and let $\cA$ be a torsion-free coherent sheaf on $X$ equipped with a singular hermitian metric $h$ with semi-positive curvature. If $\cA$ is locally free at $x$ and $\Vert \cdot \Vert_{h,x}$ is finite (hence it is a positive-definite hermitian inner product), then $\cA$ is weakly positive at $x$ .
\end{theorem}

\section{Proof of Theorem \ref{semi-negativity for harmonic bundles}}

Let $X$ be a complex manifold, and $D \subset X$ be a simple normal crossing divisor. Set $U := X \backslash D$ and $j : U \hookrightarrow X$ the inclusion. Let $(\cE, \theta , h)$ be a tame harmonic bundle on $U$, and $\cA$ be a subsheaf of $\cE$. The hermitian metric $h$ induces a (possibly singular) hermitian metric $h_{\cA}$ on $\cA$. \\

Let $\cA^{\diamond}$ be the subsheaf of $j_* \cA$ whose sections have sub-polynomial growth with respect to $h_{\cA}$. Equivalently, $ \cA^{ \diamond} := j_* \cA \cap \cE^{\diamond}$, where  $\cE^{\diamond}$ denotes the subsheaf of $j_* \cE$ whose sections have sub-polynomial growth with respect to $h_{\cE}$. As $\cE^{\diamond}$ is a locally-free sheaf of finite rank (cf. Theorem \ref{tame harmonic metric are moderate}), it follows that $\cA^{\diamond}$ is a torsion-free coherent sheaf.\\

Assume now that $\cA$ is contained in the kernel of $\theta$. It follows from the next lemma that $h_{\cA}$ has semi-negative curvature.

\begin{lem} Let $(\cE, \theta , h)$ be a harmonic bundle on a complex manifold $X$.
If $s \in \HH^0(X, \cE)$ is a section of $\cE$ which satisfies $\theta(s) =0$, then the function
\[ \log \Vert s \Vert_h : X \arrow [- \infty, + \infty ), \, x \mapsto \log \Vert s(x) \Vert_h \]
is plurisubharmonic.
\end{lem}
\begin{proof}

Denoting by $\partial_{\cE} + \overline{\partial}_{\cE}$ the Chern connection associated to the hermitian metric $h$ on the holomorphic bundle $\cE$ and by $\theta^{\star}$ the adjoint of $\theta$ with respect to $h$, we know by assumption that the connection $\nabla := \partial_{\cE} + \overline{\partial}_{\cE} + \theta + \theta^{\star}$ on the $\cC^{\infty}$-vector bundle $\cC^{\infty}_X \otimes_{\cO_X} \cE$ is flat. From this, it follows that the curvature $\Theta := (\partial_{\cE} + \overline{\partial}_{\cE} )^2 $ of $(\cE, h)$ satisfies the formula:
$$  \Theta + \theta \wedge \theta^{\star} + \theta^{\star} \wedge \theta = 0.  $$
On the other hand, if $s$ is a section of $\cE$, we have that
$$ i \cdot \partial \overline{\partial} \Vert s \Vert_h^2 = -i \cdot (\Theta s, s)_h + i \cdot \Vert \partial_{\cE} (s) \Vert_h^2 $$
(see for example \cite[(7.13)]{Schmid}). If moreover $\theta(s) = 0$, it follows that
\begin{align*}
i \cdot \partial \overline{\partial} \Vert s \Vert_h^2 \geq -i \cdot (\Theta s, s)_h &= - i \cdot (\theta(s), \theta(s))_h  + i \cdot (\theta^{\star}(s), \theta^{\star}(s))_h \\
           &= i \cdot (\theta^{\star}(s), \theta^{\star}(s))_h \\
           &\geq 0,
\end{align*}
hence the function $ \Vert s \Vert_h : X \arrow [0, + \infty ), \, x \mapsto  \Vert s(x) \Vert_h $ is plurisubharmonic.\\
Now, recall that given a non-negative valued function $v$, $\log v$ is plurisubharmonic
if and only if $v \cdot e^{2 \cdot \text{Re} (q)}$ is plurisubharmonic for every polynomial $q$ (we learned this observation from \cite{Raufi15}). The preceding computation shows that
for every polynomial $q$ the function $x \mapsto  \Vert s(x) \Vert_h^2 \cdot e^{2 \cdot \text{Re} (q(x))} = \Vert s(x) \cdot e^{ q(x)}  \Vert_h^2 $ is plurisubharmonic, because $s \cdot e^q$ is also a holomorphic section of $\cE$ which satisfies $\theta(s \cdot e^q) = e^q \cdot \theta(s) =0$. This completes the proof.
\end{proof}

Let us now prove that there exists a unique singular hermitian metric with semi-negative curvature on $\cA^{\diamond}$ whose restriction to $U$ is $h_{\cA}$. In view of Lemma \ref{extension of metric with semi-negative curvature}, this will follow from the 

\begin{lem} Let $X$ be a complex manifold, $D \subset X$ be a simple normal crossing divisor and $(\cE, \theta,h)$ be a tame harmonic bundle on $U := X \backslash D$. 
If $s$ is a section of $\cE^{\diamond}$ whose restriction to $U$ satisfies $\theta(s_{|U} ) =0$, then the function
\[ \Vert s_{|U} \Vert_h : U \arrow [0, + \infty ), \, x \mapsto \Vert s_{|U} (x) \Vert_h \]

is locally bounded in the neighborhood of every point in $X$. 

\end{lem}

\begin{proof} Clearly it is sufficient to treat the case $X = \Delta^n$ and $D = \cup_{i=1}^l D_i$, where $D_i = \{\underline{z} =(z_1, \cdots, z_n) \in \Delta^n \, | \, z_i = 0 \}$. We will show the stronger statement:

\[  \sup_{\underline{z} \in U}   \Vert s_{|U} (\underline{z}) \Vert_h \leq  \sup_{\underline{z} \in \partial X}   \Vert s_{|U} (\underline{z}) \Vert_h,  \]

where $\partial X := (\partial \Delta)^n$ is a product of circles. We claim that it is sufficient to treat the case the case $n =1$. Indeed, to bound $ \Vert s_{|U} (\underline{z}) \Vert_h$ at a point $\underline{z} \in U$, by successive applications of the case $n =1$, we get the inequalities:

\begin{align*}
 \Vert s_{|U} ((z_1, \cdots, z_n)) \Vert_h  &     \leq  \sup_{\alpha_1 \in \partial \Delta}  \Vert s_{|U} ((\alpha_1, \cdots, z_n)) \Vert_h \\
           & \leq  \sup_{\alpha_1 \in \partial \Delta, \alpha_2 \in \partial \Delta}  \Vert s_{|U} ((\alpha_1, \alpha_2, \cdots, z_n)) \Vert_h  \\
	& \cdots \\        
	& \leq  \sup_{\underline{\alpha} \in (\partial \Delta)^n}  \Vert s_{|U} ((\alpha_1, \alpha_2, \cdots, \alpha_n)) \Vert_h .
\end{align*}

It remains to prove the case $n =1$. We will first show as a consequence of Simpson's norm estimates that the function $ \Vert s_{|U} \Vert_h : U \arrow [0, + \infty ), x \mapsto \Vert s_{|U} (x) \Vert_h$ is locally bounded. It follows that the function $\log  \Vert s_{|U} \Vert_h$ extends as a plurisubharmonic function to $\Delta$, and we conclude using the maximum principle. \\

Finally, we have reduced the proof to the case where $X = \Delta$ and $D = \{ 0 \}$.
Let $R := \Res(\cE^{\diamond})$ be the residue of $\theta$ at $0$. It is a $\bC$-linear endomorphism of the complex vector space $\cE^{\diamond}_{|0}$.
Let $\cE^{\diamond}_{|0} = \bigoplus_{\lambda \in \bC} ({\cE^{\diamond}_{|0}})_{\lambda}$ be the decomposition in generalized eigenspaces of $R$. On each piece $  ({\cE^{\diamond}_{|0}})_{\lambda}$ the nilpotent
endomorphism $R - \lambda \cdot \Id $ defines a weight filtration $W_k ( ({\cE^{\diamond}_{|0}})_{\lambda})$ indexed by integers $k$, thanks to the well-known:

\begin{lem}[cf. {\cite[lemma 6.4]{Schmid}}] Let $N$ be a nilpotent endomorphism of a finite dimensional complex vector space $V$. There exists a unique finite increasing filtration $W_k (V)$ of $V$ such that:
\begin{enumerate}
\item $  N (W_k ) \subset W_{k-2}$ for every $k$,
\item $N^k$ induces an isomorphism $Gr_k^W V \xrightarrow{\sim} Gr_{-k}^W V$ for every $k \geq 0$. 
\end{enumerate}
\end{lem}
 
The sum of the weight filtrations for different eigenvalues $\lambda$ yields in turn a filtration of $ \cE^{\diamond}_{|0}$ that we still denote by $W_k$. As proved by Schmid \cite{Schmid} for variations of Hodge structures and Simpson \cite{Simpson_open} for general tame harmonic bundles, this filtration controls the growths of the sections of $\cE^{\diamond}$:\\
If $s$ be a section of $\cE^{\diamond}$ defined on $\Delta$, then $s(0) \in W_k$ if and only if there exists a constant $C>0$ such that the following estimates holds (cf. \cite[pp. 755-756]{Simpson_open}):
\[   \Vert s(z) \Vert_h \leq C \cdot  (- \log |z|)^{k/2} . \]

Let $s$ be a section of $\cE^{\diamond}$ defined on $\Delta$ such that $\theta(s) =0$. It follows that $s(0)$ belongs to the kernel of $\Res(\cE^{\diamond})$. In particular, it belongs to the generalized eigenspace where $\Res(\cE^{\diamond})$ is nilpotent. Moreover, by the lemma below, it belongs to $W_0$ and by the norm estimates it follows that $ \Vert s \Vert_h$ is bounded on $\Delta^{\star}$.

\begin{lem}[cf. {\cite[corollary 6.7]{Schmid}}] Let $N$ be a nilpotent endomorphism of a finite dimensional complex vector space $V$, and let $W_k (V)$ be the corresponding filtration of $V$.
Then $\ker( N) \subset W_0 (V)$.
\end{lem}
\end{proof}

\section{Proofs of Theorem \ref{main result} and Theorem \ref{trivial_parabolic_filtration}}

Let $X$ be a smooth projective complex variety, $\cL$ be an ample line bundle on $X$ and $D \subset X$ be a simple normal crossing divisor. Let $(\underline {\cE}, \theta)$ be a filtered regular meromorphic Higgs bundle on $(X,D)$, which is $\mu_{\cL}$-polystable with vanishing parabolic Chern classes. Let finally $\cA$ be a subsheaf of $\cE^{\diamond}$ contained in the kernel of $\theta$. Set $U := X \backslash D$ and $j : U \hookrightarrow X$ the inclusion. By Theorem \ref{Mochizuki correspondence}, there exists an essentially unique pluriharmonic metric $h$ on $(\cE^{\diamond}_{|U}, \theta)$ adapted to the parabolic structure. The triple $(\cE^{\diamond}_{|U}, \theta, h)$ defines a tame harmonic bundle on $U$, and $\cE^{\diamond}$ is equal to the subsheaf of $j_*({\cE}^{\diamond}_{|U})$ whose sections have sub-polynomial growth with respect to $h$. Observe also that $\cA$ being torsion-free, the canonical map $\cA \arrow j_* (\cA_{|U}) \cap  \cE^{\diamond}$ is an injective map of sheaves and an isomorphism when restricted to $U$. By Theorem \ref{semi-negativity for harmonic bundles} the (possibly singular) hermitian metric induced by $h$ on $\cA_{|U}$ extends uniquely as a singular hermitian metric with semi-negative curvature on $ j_* (\cA_{|U}) \cap  \cE^{\diamond}$. By restriction we obtain a singular hermitian metric with semi-negative curvature on $\cA$ that we denote $h_{\cA}$. This metric induces in turn a singular hermitian metric with semi-positive curvature $h_{\cA^{\star}}$ on $\cA^{\star}$. One can then apply Theorem \ref{Paun-Takayama} to conclude that its dual ${\cA}^{\star}$ is weakly positive. In fact one obtains the more precise

\begin{theorem}\label{thm_restricted_base_locus}
Let $X$ be a smooth projective complex variety, $\cL$ be an ample line bundle on $X$ and $D \subset X$ be a simple normal crossing divisor. Let $(\underline {\cE}, \theta)$ be a filtered regular meromorphic Higgs bundle on $(X,D)$. Assume that $(\underline {\cE}, \theta)$ is $\mu_{\cL}$-polystable with vanishing parabolic Chern classes. If $\cA$ is a subsheaf of $\cE^{\diamond}$ contained in the kernel of $\theta$, then its dual ${\cA}^{\star}$ is weakly positive in the sense of Viehweg, and its restricted base locus $\bB_-(\cA^{\star})$ is contained in the union of $ D$ and the locus in $X$ where $ \cA$ is not a locally split subsheaf of $\cE^{\diamond} $. 
\end{theorem}

Assume from now on that $\cA$ is a locally split subsheaf of $\cE$. Let $\pi : \bP({\cA}^{\star}) \arrow X$ be the projective scheme over $X$ associated to ${\cA}^{\star}$, i.e. $\bP({\cA}^{\star}) := \text{Proj}_{\cO_X} (\Sym({\cA}^{\star}))$. Let $\cO_{{\cA}^{\star}}(1)$ be the tautological line bundle on $\bP({\cA}^{\star})$, so that there is an exact sequence $\pi^* {\cA}^{\star} \arrow   \cO_{{\cA}^{\star}}(1) \arrow 0$. The singular hermitian metric $h_{\cA^{\star}}$ with semi-positive curvature on ${\cA}^{\star}$ induces canonically a singular hermitian metric $h_{\cO_{{\cA}^{\star}}(1)}$ with semi-positive curvature on $\cO_{{\cA}^{\star}}(1)$.

\begin{prop}\label{prop_Lelong_numbers}
The Lelong numbers of the metric $h_{\cO_{{\cA}^{\star}}(1)}$ are smaller than $1$. In particular, the multiplier ideal $\cI(h_{\cO_{{\cA}^{\star}}(1)})$ is trivial. If moreover the parabolic structure on $\underline {\cE}$ is trivial, then all Lelong numbers are zero (see below for the definition of Lelong numbers).
\end{prop}

\begin{defin}[compare {\cite[Definition 2.5]{Demailly_singular_metric}}]
On a complex manifold $X$, let $\cL$ be a line bundle equipped with a singular hermitian metric $h$ of semi-positive curvature. The \textit{Lelong number} of $h$ at the point $x \in X$ is defined by 
\[ \nu(h, x) := \liminf_{z \arrow x} \frac{ \log \Vert s(z) \Vert_h}{- \log |z - x|}, \]
 where $s$ is a holomorphic section of $\cL$ defined in a neighborhood of $x$ with $s(x) \neq 0$.
\end{defin}

One verifies that it is a well-defined non-negative real number which does not depend on the local section $s$, cf. \textit{loc. cit.}

\begin{proof}[Proof of Proposition \ref{prop_Lelong_numbers}]
From the exact sequences of hermitian holomorphic vector bundles $\pi^* {\cA}^{\star} \arrow   \cO_{{\cA}^{\star}}(1) \arrow 0$ and $(\cE^{\diamond})^{\star} \arrow {\cA}^{\star} \arrow 0$ one sees that the order of growth of $\pi^* {\cA}^{\star}$ with respect to $\pi^* h_{{\cA}^{\star}}$ is bigger or equal than the order of growth of $(\cE^{\diamond})^{\star}$ with respect to $h^{\star}$. Let us show that the latter is bigger than $-1$. First note that $(\cE^{\diamond})^{\star}$ is canonically isomorphic to $(\underline{\cE}^{\star})^{> -1}$, where $\underline{\cE}^{\star}$ denotes the dual of the parabolic bundle $\underline{\cE}$ and $(\underline{\cE}^{\star})^{> -1}$ denotes the subsheaf of the parabolic bundle $\underline{\cE}^{\star}$ corresponding to the weight $(- 1 + \epsilon, \cdots, -1 + \epsilon)$ for $\epsilon >0$ sufficiently small. This follows immediately from the definition of the dual filtration: $\lambda \in (\underline{\cE}^{\star})^{\alpha}$ if and only if $\lambda ( \cE^{\beta}) \subset \cO_X^{\alpha + \beta}$ for all $\beta$. On the other hand, the filtration according to the order of growth is easily seen to be compatible with taking duals, hence $(\underline{\cE}^{\star})^{> -1}$ is the subsheaf of $j_* (\cE_{|U}^{\star})$ of sections whose order of growth with respect to $h^{\star}$ is bigger than $-1$ along every irreducible component of $D$. From this discussion, it follows that the Lelong numbers of the metric $h_{\cO_{{\cA}^{\star}}(1)}$ are smaller than $1$. In particular, all holomorphic sections of $\cO_{{\cA}^{\star}}(1)$ are locally $L^2$ with respect to $h_{\cO_{{\cA}^{\star}}(1)}$ (cf. \cite[Lemma 2.8]{Demailly_singular_metric}), meaning that the multiplier ideal $\cI(h_{\cO_{{\cA}^{\star}}(1)})$ is trivial.  \\

When the parabolic structure on $\underline {\cE}$ is trivial, one sees that $(\cE^{\diamond})^{\star}$ is canonically isomorphic to $(\underline{\cE}^{\star})^{\diamond}$, hence the Lelong numbers of the metric $h_{\cO_{{\cA}^{\star}}(1)}$ are all zero.
\end{proof}

Finally, Theorem \ref{trivial_parabolic_filtration} is a consequence of the preceding Proposition \ref{prop_Lelong_numbers} and the

\begin{lem} Let $X$ be a smooth projective variety and $\cL$ be a line bundle on $X$ equipped with a singular metric $h$ with semi-positive curvature. If the Lelong numbers of $(\cL, h)$ satisfy $\nu(h, x) = 0$ for all $x \in X$ but a countable set, then $\cL$ is nef.
\end{lem}
\begin{proof}
 Follows from \cite[Corollary 6.4]{Demailly_regularization} and  \cite[Proposition 6.1]{Demailly_regularization}.
\end{proof}

\section{Variations of Hodge structures and proof of Theorem \ref{thm_PVHS}}\label{PVHS}

We begin by recalling some definitions.

\begin{defin}
Let $V$ be a complex vector space of finite dimension. A (complex polarized) Hodge structure (of weight zero) on $V$ is the data of a non-degenerate hermitian form $h$ and a $h$-orthogonal decomposition
$ V = \bigoplus_{p \in \bZ} V^p $
such that the restriction of $h$ to $V^p$ is positive definite for $p$ even and negative definite for $p$ odd. 
The associated Hodge filtration is the decreasing filtration $F$ on $V$ defined by 
$ F^p := \bigoplus_{ q \geq p} V^q$.
\end{defin}

\begin{defin}
Let $X$ be a complex manifold. A variation of polarized complex Hodge structures ($\bC$-PVHS) on $X$ is the data of a holomorphic vector bundle $\mathscr{E}$ equipped with an integrable connection $\nabla$, a $\nabla$-flat non-degenerate hermitian form $h$ and for all $x \in X$ a decomposition of the fibre $\mathscr{E}_x = \bigoplus_{p \in \bZ} \mathscr{E}^p$ satisfying the following axioms:

\begin{itemize}
\item for all $x \in X$, the decomposition $\mathscr{E}_x = \bigoplus_{p \in \bZ} \mathscr{E}_x^p $ defines a Hodge structure polarized by $h_x$,
\item  the Hodge filtration $\cF$ varies holomorphically with $x$, 
\item (Griffiths' transversality) $\nabla( \cF^p) \subset \cF^{p-1} \otimes_{\cO_X} \Omega_X^1$ for all $p$.
\end{itemize}  
\end{defin}

One obtains a positive-definite hermitian metric $h_H$ on $\mathscr{E}$ from $h$ by imposing that for all $x \in X$ the decomposition $ \mathscr{E}_x= \bigoplus_{p \in \bZ} \mathscr{E}_x^p $ is $h_H$-orthogonal and setting $h_H := (-1)^p \cdot h$ on $\mathscr{E}_x^p$. We call $h_H$ the Hodge metric.

\begin{defin}[Log $\bC$-PVHS] Let $X$ be a complex manifold, and $D \subset X$ be a simple normal crossing divisor.  A log complex polarized variation of Hodge structure (log $\bC$-PVHS) on $(X,D)$ consists of the following data:
\begin{itemize}
\item A holomorphic vector bundle $\mathscr{E}$ on $X$ endowed with a connection $\nabla$ with logarithmic singularities along $D$,
\item An exhaustive decreasing filtration $\cF$ on $\mathscr{E}$ by holomorphic subbundles (the Hodge filtration), satisfying Griffiths transversality
\[ \nabla \cF^p \subset \cF^{p-1} \otimes \Omega^1_X(\log D), \]
\item a $\nabla$-flat non-degenerate hermitian form $h$ on $\mathscr{E}_{|X \backslash D}$,
\end{itemize}
such that $(\mathscr{E}_{| X \backslash D}, \nabla, \cF^{\sbt}_{|X \backslash D}, h)$ is a $\bC$-PVHS on $X \backslash D$.
\end{defin}

Note that by setting  $\cE := \Gr_{\cF} \mathscr{E}$ and $ \theta  := \Gr_{\cF} \nabla$ we get a logarithmic Higgs bundle on $(X,D)$, called the associated logarithmic Higgs bundle.\\

The following result explains the link between $\bC$-PVHS and harmonic bundles.

\begin{lem}[Simpson, {\cite[section 8]{Simpson_constructing}}]\label{Harmonic_bundle_from_PVHS} If $ (\mathscr{E}, \nabla ,  \cF^{\sbt}, h)$ is a $\bC$-PVHS on a complex manifold, then the triplet $ (\mathscr{E}, \nabla, h_H)$ forms a harmonic bundle, where $h_H$ denotes the Hodge metric on $\mathscr{E}$. The corresponding Higgs bundle $(\cE, \theta)$ is given by $\cE := \Gr_{\cF} \mathscr{E}$ and $ \theta  := \Gr_{\cF} \nabla$.
\end{lem}

Let $X$ be a complex manifold, and $D = \cup_{i \in I} D_i \subset X$ be a simple normal crossing divisor. Let $(\mathscr{E}, \nabla ,  \cF^{\sbt}, h)$ be a $\bC$-PVHS on $U := X  \backslash D$.
First note that the harmonic bundle $ (\mathscr{E}, \nabla, h_H)$ is tame. Indeed, if $(\EE = \bigoplus_{p \in \bZ}  \EE^p, \theta)$ is the associated Higgs bundle, then one sees imediately that $\theta$ is nilpotent, and as a consequence its characteristic polynomial $T^{\dim U}$ extends clearly to $X$ as an element of $\Sym (\Omega_X^1(\log D)) \left[T \right]$.

There exists a unique meromorphic bundle on $(X,D)$ equipped with a regular meromorphic connection extending $(\mathscr{E}, \nabla)$ (the so-called Deligne's extension). When equipped with the Deligne-Manin filtration (cf. section \ref{Deligne-Manin filtration}), this defines a canonical filtered regular meromorphic connection bundle $(\underline{\mathscr{E}}^{DM}, \nabla)$ on $(X,D)$ which extends $(\mathscr{E}, \nabla)$. As $\theta$ is nilpotent, the eigenvalues of its residues along the $D_i$ are zero, hence by Lemma \ref{DM=h} $(\underline{\mathscr{E}}^{DM}, \nabla)$  is canonically isomorphic to the filtered regular meromorphic connection bundle $(\underline{\mathscr{E}}^{h}, \nabla)$ on $(X,D)$ defined using the filtration according to the order of growth with respect to $h$. We denote by $\mathscr{E}^{\alpha}$ the elements of this filtration. We can extend the Hodge filtration to every $\mathscr{E}^{\alpha}$ by setting:
\begin{equation*}
\cF^p \mathscr{E}^{\alpha}  := \mathscr{E}^{ \alpha} \cap j_* \cF^p.
\end{equation*}

For every $p$, we have an exact sequence of locally-free sheaves on $U$:
\[ 0 \arrow \cF^{p+1}  \mathscr{E} \arrow \cF^p  \mathscr{E}  \arrow Gr_{\cF}^p  \mathscr{E} = \cE^p \arrow 0. \]

The Hodge metric on $\mathscr{E}$ induces canonical hermitian metrics on these locally-free sheaves. From this it follows that for every $\alpha$ and $p$ one get an exact sequence of sheaves on $X$:
\[ 0 \arrow \cF^{p+1}  \mathscr{E}^{\alpha}  \arrow \cF^p  \mathscr{E}^{\alpha}   \arrow {(\cE^p)}^{\alpha}. \]

Consequently, there is a canonical injective map of sheaves $ g_{\alpha} : Gr_{\cF}  \mathscr{E}^{\alpha} \arrow  {\cE}^{\alpha}$ for every $\alpha$.
Note however that it is not clear a priori that these maps are surjective (when the residues of $(\mathscr{E}, \nabla)$ are nilpotent, it is a consequence of the work of Schmid \cite{Schmid}). Even if this fact is not strictly needed in our discussion, we will give a proof for the sake of completeness.\\

First note that the preceding maps $g_{\alpha}$ are isomorphisms if and only if the induced maps on the determinants $ \det g_{\alpha} : \det (Gr_{\cF}  \mathscr{E}^{\alpha}) \arrow \det ({\cE}^{\alpha})$ are isomorphisms. But $ \det Gr_{\cF}  \mathscr{E}^{\alpha}$ is canonically isomorphic to $ \det  \mathscr{E}^{\alpha}$ and the filtration according to growth conditions is compatible with taking determinant, hence this reduces the problem to show that the canonical maps $(\det \mathscr{E})^{\alpha} \arrow (\det {\cE})^{\alpha}$ are isomorphisms for every $\alpha$. This is clear because in restriction to $U$ the map $(\det \mathscr{E}) \arrow (\det {\cE})$ is an isomorphism of hermitian holomorphic vector bundles. Therefore we have proved that the maps $ g_{\alpha} : Gr_{\cF}  \mathscr{E}^{\alpha} \arrow  {\cE}^{\alpha}$ are isomorphisms of sheaves (in particular, the $ \cF^p  \mathscr{E}^{\alpha} $ form a filtration of $  \mathscr{E}^{\alpha} $ by locally split subsheaves). \\

We can now give a proof of the Theorem \ref{thm_PVHS}. With the notations of the statement, it follows from the discussion above that there exists an injective map of sheaves $0 \arrow \cA \arrow {\cE}^{\diamond}$. Moreover, $\cA$ is in the kernel of the Higgs field. We get the result by applying Theorem \ref{semi-negativity for harmonic bundles} to this particular situation.


\section{Proof of Theorem \ref{generically large local system}}
In this section we prove Theorem \ref{generically large local system} as a direct consequence of the two following propositions. 
\begin{prop}
Let $X$ be a smooth projective complex variety, $D \subset X$ be a simple normal crossing divisor and $(\cE, \theta, h)$ be a tame harmonic bundle on $U := X \backslash D$. If the map of $\cO_U$-modules $\phi : T_U \arrow \End(\cE)$ associated to $\theta \in \Omega^1_U(\End(\cE))$ is injective, then the logarithmic cotangent bundle $\Omega^1_X(\log D)$ of $(X,D)$ is weakly positive. 
\end{prop}

\begin{proof}

Let $(\underline{\cE}, \theta) = (\underline{\cE}^{h}, \theta)$ be the filtered regular meromorphic Higgs bundle on $(X,D)$ associated to $(\cE, \theta, h)$ using the filtration according to the order of growth with respect to $h$, cf. Theorem \ref{tame harmonic metric are moderate}. The parabolic bundle $\End (\underline{\cE})$ inherits a canonical structure of filtered regular meromorphic Higgs bundle, whose Higgs field $\Theta$ is defined by
\begin{align*}
(\Theta_s( \Psi ))(v) = \theta_s ( \Psi (v)) - \Psi( \theta_s (v)) 
\end{align*}
for $\Psi $ a local holomorphic section of $\End(\cE)$, $v$ a local holomorphic section of $\cE$ and $s$ a local holomorphic section of $T_U$. Here we denote by $\theta_s $ the contraction of $\theta$ with $s$, alias $\phi(s)$ (and similarly for $\Theta_s$). \\ 

The composition of $\phi : T_U \arrow \End(\cE)$ with the Higgs field $\Theta : \End(\cE) \arrow  \Omega^1_U \otimes_{\mathscr{O}_U} \End(\cE)$ is zero: if $s$ and $t$ are local holomorphic sections of $ T_U$ and $v$ is a local holomorphic section of $\cE$, then the condition $\theta \wedge \theta = 0$ implies
 \begin{align*}
(\Theta_s( \theta_t ))(v) = \theta_s ( \theta_t (v)) - \theta_t( \theta_s (v)) = 0, 
\end{align*}
where $\phi(t) = \theta_t \in \End(\cE)$. As a consequence, if we denote by $\cA \subset \End (\cE)^{\diamond}$ the image of $\phi : T_X(- \log D) \arrow \End (\cE)^{\diamond} $, then $\cA$ is contained in the kernel of the Higgs field $\Theta :  \End(\cE)^{\diamond}  \arrow  \Omega^1_X(\log D) \otimes_{\mathscr{O}_X} \End(\cE)^{\diamond} $.\\

If $\cL$ is an ample line bundle on $X$, then by Theorem \ref{Mochizuki correspondence} the filtered regular meromorphic Higgs bundle $(\underline {\cE}, \theta)$ is $\mu_{\cL}$-polystable with vanishing parabolic Chern classes, and the same is true for $(\End(\underline {\cE}), \Theta)$. The subsheaf $\cA \subset \End(\cE)^{\diamond} $ is contained in the kernel of $\Theta$, hence by Theorem \ref{main result} its dual ${\cA}^{\star}$ is weakly positive. Since the map $ T_X(- \log D) \arrow \cA $ is generically an isomorphism by assumption, the same is true for the dual map ${\cA}^{\star} \arrow \Omega^1_X(\log D)$. It follows that $\Omega^1_X(\log D)$ is weakly positive.
\end{proof}

Let $X$ be a smooth projective complex variety and $D \subset X$ be a simple normal crossing divisor. A tame harmonic bundle $(\mathscr{E}, \nabla, h) \equiv (\cE, \theta, h)$ is called purely imaginary if the eigenvalues of the residues of $\theta$ along the irreducible components of $D$ are purely imaginary. It is equivalent to ask that the filtered regular meromorphic connection bundles $(\underline{\mathscr{E}}^{DM}, \nabla)$ and $(\underline{\mathscr{E}}^{h}, \nabla)$ are canonically isomorphic, cf. Lemma \ref{DM=h}. Recall that a complex local system on $X \backslash D$ is semisimple if and only if the corresponding connection bundle comes from a purely imaginary tame harmonic bundles on $X \backslash D$ \cite[Theorem 25.28]{Mochizuki_AMSII}. Also, if $X$ and $Y$ are smooth projective complex varieties, $D_X$ and $D_Y$ are simple normal crossing divisors of $X$ and $Y$ respectively and $f : X \arrow Y$ is a morphism such that $f^{-1}(D_Y) \subset D_X$, then 
for any tame harmonic bundle $(\cE,\theta ,h)$ on $Y \backslash D_Y$, its pullback $f^*(\cE,\theta ,h)$ is a tame harmonic bundle on $X \backslash D_X$. Moreover, if $(\cE,\theta ,h)$ is pure imaginary, then $f^* (\cE,\theta ,h)$ is also pure imaginary, cf. \cite[Lemma 25.29]{Mochizuki_AMSII} .

\begin{prop}
Let $X$ be a smooth projective complex variety, $D \subset X$ be a simple normal crossing divisor and $(\cE, \theta, h)$ be a pure imaginary tame harmonic bundle on $U := X \backslash D$. If the associated semisimple complex local system on $U$ is generically large and has discrete monodromy, then the map of $\cO_U$-modules $\phi : T_U \arrow \End(\cE)$ associated to $\theta \in \Omega^1_U(\End(\cE))$ is injective. 
\end{prop}
Before giving the proof of the proposition, let us prove the following easy lemma.

\begin{lem}
Let $(\cE, \theta, h)$ be a pure imaginary tame harmonic bundle on a smooth quasi-projective complex variety $U := X \backslash D$, and let $L$ be the associated semisimple complex local system on $U$. Then the monodromy of $L$ is finite if and only if it is discrete and the Higgs field $\theta$ is zero.
\end{lem}
\begin{proof}
If the monodromy of $L$ is finite, then the pullback of $L$ to a finite \'etale cover of $U$ is trivial. Therefore the pullback of the Higgs field by a finite \'etale cover is zero, hence $\theta$ is zero itself. For the converse, first observe that given a harmonic bundle $(\cE, \theta, h)$ on a complex manifold, the Higgs field $\theta$ is zero if and only if the pluriharmonic metric $h$ is parallel for the Chern connection of $(\cE,h)$. Therefore, the monodromy being unitary and discrete, it has to be finite.
\end{proof}

\begin{proof}[Proof of the Proposition]
By assumption, there exist countably many closed subvarieties $D_i \subsetneq U$ such that for every smooth quasi-projective complex variety $Z$ equipped with a proper map $f : Z \arrow U$ satisfying $f(Z) \not \subset \cup D_i$ the pullback local system $f^\ast L$ is non-trivial (and has in fact infinite monodromy). Let also $Y \subsetneq U$ be the closed subvariety where the cokernel of the map of $\cO_U$-modules $\phi : T_U \arrow \End(\cE)$ is not locally free. \\
Fix $x \in X \backslash (\cup D_i \cup Y)$ and let us show that $\phi (s) \neq 0$ for all $s \in T_X(x) \backslash \{ 0 \}$. Let $C \subset U$ be a complete intersection curve containing $x$ with $ T_C(x) = \bC \cdot s \subset  T_X(x)$, let $\nu : \tilde{C} \arrow C$ be its normalization and $f : \tilde{C} \arrow X$ be the composition of $\nu$ with the inclusion $C \subset X$. Since $f(\tilde{C}) = C \not \subset \cup D_i$, the monodromy of the local system $f^\ast L$ is infinite and discrete. By the preceding Lemma, it follows that the Higgs field of the tame harmonic bundle $f^*(\cE,\theta ,h)$ is nonzero. In other words, the composition $T_C \arrow f^{\ast} T_U \arrow f^{\ast}  \End(\cE)$ is nonzero. Since in restriction to a neighborhood of $x \in C$ the cokernel of the map $T_C \arrow f^{\ast}  \End(\cE)$ is locally free, the image of $v$ is necessarily nonzero, i.e. $\phi(s) \neq 0$.

\end{proof}


\bibliographystyle{alpha}
\bibliography{biblio}

\begin{thebibliography}{{Ara}16}

\bibitem[{Ara}16]{Arapura16}
D.~{Arapura}.
\newblock {Kodaira-Saito vanishing via Higgs bundles in positive
  characteristic}.
\newblock {\em ArXiv e-prints}, November 2016.

\bibitem[BC17]{Brunebarbe-Cadorel}
Y.~{Brunebarbe} and B.~{Cadorel}.
\newblock {Hyperbolicity of varieties supporting a variation of Hodge
  structure}.
\newblock {\em ArXiv e-prints}, July 2017.

\bibitem[Biq97]{Biquard}
Olivier Biquard.
\newblock Fibr\'es de {H}iggs et connexions int\'egrables: le cas logarithmique
  (diviseur lisse).
\newblock {\em Ann. Sci. \'Ecole Norm. Sup. (4)}, 30(1):41--96, 1997.

\bibitem[BP08]{Berndtsson-Paun}
Bo~Berndtsson and Mihai P{\u{a}}un.
\newblock Bergman kernels and the pseudoeffectivity of relative canonical
  bundles.
\newblock {\em Duke Math. J.}, 145(2):341--378, 2008.

\bibitem[Bru]{Brunebarbe_Crelle}
Yohan Brunebarbe.
\newblock Symmetric differentials and variations of {H}odge structures.
\newblock {\em To appear in Crelle's Journal}.

\bibitem[CCE15]{CCE15}
Fr\'ederic Campana, Beno\^\i~t Claudon, and Philippe Eyssidieux.
\newblock Repr\'esentations lin\'eaires des groupes k\"ahl\'eriens:
  factorisations et conjecture de {S}hafarevich lin\'eaire.
\newblock {\em Compos. Math.}, 151(2):351--376, 2015.

\bibitem[Del70]{Deligne-singuliers-reguliers}
Pierre Deligne.
\newblock {\em \'{E}quations diff\'erentielles \`a points singuliers
  r\'eguliers}.
\newblock Lecture Notes in Mathematics, Vol. 163. Springer-Verlag, Berlin-New
  York, 1970.

\bibitem[Dem92a]{Demailly_regularization}
Jean-Pierre Demailly.
\newblock Regularization of closed positive currents and intersection theory.
\newblock {\em J. Algebraic Geom.}, 1(3):361--409, 1992.

\bibitem[Dem92b]{Demailly_singular_metric}
Jean-Pierre Demailly.
\newblock Singular {H}ermitian metrics on positive line bundles.
\newblock In {\em Complex algebraic varieties ({B}ayreuth, 1990)}, volume 1507
  of {\em Lecture Notes in Math.}, pages 87--104. Springer, Berlin, 1992.

\bibitem[FF14]{Fujino-Fujisawa}
Osamu Fujino and Taro Fujisawa.
\newblock Variations of mixed {H}odge structure and semipositivity theorems.
\newblock {\em Publ. Res. Inst. Math. Sci.}, 50(4):589--661, 2014.

\bibitem[FF17]{Fujino-Fujisawa2017}
O.~{Fujino} and T.~{Fujisawa}.
\newblock {On semipositivity theorems}.
\newblock {\em ArXiv e-prints}, January 2017.

\bibitem[FFS14]{Fujino-Fujisawa-Saito}
Osamu Fujino, Taro Fujisawa, and Morihiko Saito.
\newblock Some remarks on the semipositivity theorems.
\newblock {\em Publ. Res. Inst. Math. Sci.}, 50(1):85--112, 2014.

\bibitem[Fuj78]{Fujita78}
Takao Fujita.
\newblock On {K}\"ahler fiber spaces over curves.
\newblock {\em J. Math. Soc. Japan}, 30(4):779--794, 1978.

\bibitem[HPS16]{Hacon-Popa-Schnell}
C.~{Hacon}, M.~{Popa}, and C.~{Schnell}.
\newblock {Algebraic fiber spaces over abelian varieties: around a recent
  theorem by Cao and Paun}.
\newblock {\em ArXiv e-prints}, November 2016.

\bibitem[IS08]{Iyer_Simpson}
Jaya~N. Iyer and Carlos~T. Simpson.
\newblock The {C}hern character of a parabolic bundle, and a parabolic
  corollary of {R}eznikov's theorem.
\newblock In {\em Geometry and dynamics of groups and spaces}, volume 265 of
  {\em Progr. Math.}, pages 439--485. Birkh\"auser, Basel, 2008.

\bibitem[Kaw81]{Kawamata81}
Yujiro Kawamata.
\newblock Characterization of abelian varieties.
\newblock {\em Compositio Math.}, 43(2):253--276, 1981.

\bibitem[Kol93]{Kollar93}
J\'anos Koll\'ar.
\newblock Shafarevich maps and plurigenera of algebraic varieties.
\newblock {\em Invent. Math.}, 113(1):177--215, 1993.

\bibitem[Kol95]{Kollar_book_95}
J\'anos Koll\'ar.
\newblock {\em Shafarevich maps and automorphic forms}.
\newblock M. B. Porter Lectures. Princeton University Press, Princeton, NJ,
  1995.

\bibitem[Moc06]{Mochizuki_asterisque309}
Takuro Mochizuki.
\newblock Kobayashi-{H}itchin correspondence for tame harmonic bundles and an
  application.
\newblock {\em Ast\'erisque}, (309):viii+117, 2006.

\bibitem[Moc07]{Mochizuki_AMSII}
Takuro Mochizuki.
\newblock Asymptotic behaviour of tame harmonic bundles and an application to
  pure twistor {$D$}-modules. {II}.
\newblock {\em Mem. Amer. Math. Soc.}, 185(870):xii+565, 2007.

\bibitem[Moc09]{Mochizuki_GeoTopo}
Takuro Mochizuki.
\newblock Kobayashi-{H}itchin correspondence for tame harmonic bundles. {II}.
\newblock {\em Geom. Topol.}, 13(1):359--455, 2009.

\bibitem[Moc11]{Mochizuki_Wild-harmonic-bundles}
Takuro Mochizuki.
\newblock Wild harmonic bundles and wild pure twistor {$D$}-modules.
\newblock {\em Ast\'erisque}, (340):x+607, 2011.

\bibitem[Mok92]{Mok92}
Ngaiming Mok.
\newblock Factorization of semisimple discrete representations of {K}\"ahler
  groups.
\newblock {\em Invent. Math.}, 110(3):557--614, 1992.

\bibitem[Nak04]{Nakayama}
Noboru Nakayama.
\newblock {\em Zariski-decomposition and abundance}, volume~14 of {\em MSJ
  Memoirs}.
\newblock Mathematical Society of Japan, Tokyo, 2004.

\bibitem[P{\u a}u16]{Paun16}
M.~P{\u a}un.
\newblock {Singular Hermitian metrics and positivity of direct images of
  pluricanonical bundles}.
\newblock {\em ArXiv e-prints}, June 2016.

\bibitem[PT14]{Paun-Takayama}
M.~P{\u a}un and S.~{Takayama}.
\newblock {Positivity of twisted relative pluricanonical bundles and their
  direct images}.
\newblock {\em ArXiv e-prints}, September 2014.

\bibitem[PW16]{Popa_Wu}
Mihnea Popa and Lei Wu.
\newblock Weak positivity for {H}odge modules.
\newblock {\em Math. Res. Lett.}, 23(4):1139--1155, 2016.

\bibitem[Rau15]{Raufi15}
Hossein Raufi.
\newblock Singular hermitian metrics on holomorphic vector bundles.
\newblock {\em Ark. Mat.}, 53(2):359--382, 2015.

\bibitem[Sab05]{Sabbah05}
Claude Sabbah.
\newblock Polarizable twistor {$\mathscr D$}-modules.
\newblock {\em Ast\'erisque}, (300):vi+208, 2005.

\bibitem[Sch73]{Schmid}
Wilfried Schmid.
\newblock Variation of {H}odge structure: the singularities of the period
  mapping.
\newblock {\em Invent. Math.}, 22:211--319, 1973.

\bibitem[Sim88]{Simpson_constructing}
Carlos~T. Simpson.
\newblock Constructing variations of {H}odge structure using {Y}ang-{M}ills
  theory and applications to uniformization.
\newblock {\em J. Amer. Math. Soc.}, 1(4):867--918, 1988.

\bibitem[Sim90]{Simpson_open}
Carlos~T. Simpson.
\newblock Harmonic bundles on noncompact curves.
\newblock {\em J. Amer. Math. Soc.}, 3(3):713--770, 1990.

\bibitem[Zuc82]{Zucker82}
Steven Zucker.
\newblock Remarks on a theorem of {F}ujita.
\newblock {\em J. Math. Soc. Japan}, 34(1):47--54, 1982.

\bibitem[Zuo96]{Zuo96}
Kang Zuo.
\newblock Kodaira dimension and {C}hern hyperbolicity of the {S}hafarevich maps
  for representations of {$\pi_1$} of compact {K}\"ahler manifolds.
\newblock {\em J. Reine Angew. Math.}, 472:139--156, 1996.

\bibitem[Zuo00]{Zuo_neg}
Kang Zuo.
\newblock On the negativity of kernels of {K}odaira-{S}pencer maps on {H}odge
  bundles and applications.
\newblock {\em Asian J. Math.}, 4(1):279--301, 2000.
\newblock Kodaira's issue.

\end{thebibliography}

\vspace{0.5cm}

\textsc{Yohan Brunebarbe, Institut f\"ur Mathematik, Universit\"at Z\"urich, Winterthurerstrasse 190, CH-8057 Z\"urich, Schweiz} \par\nopagebreak
  \textit{E-mail address}: \texttt{yohan.brunebarbe@math.uzh.ch}

\end{document}